\documentclass[12pt]{article}
 \usepackage{geometry} 
\usepackage{amsmath,amsthm,amssymb,amsfonts,tikz,mathtools, breqn}
\usetikzlibrary{calc,decorations.pathreplacing,arrows}

\newcommand{\Pb}{\mathbb{P}}
\newcommand{\E}{\mathbb{E}}
\newcommand\pfun{\mathrel{\ooalign{\hfil$\mapstochar\mkern5mu$\hfil\cr$\to$\cr}}}

\newtheorem{letterthm}{Theorem}

\newtheorem{theorem}{Theorem}
\newtheorem{lemma}[theorem]{Lemma}

\newtheorem{proposition}[theorem]{Proposition}
\theoremstyle{definition}
\newtheorem{definition}[theorem]{Definition}
\theoremstyle{remark}

\newtheorem{example}[theorem]{Example}

\numberwithin{theorem}{section}
 
\begin{document}
\title{Separability and Randomness in Free Groups}
\author{Michal Buran \footnote{Supported by ERC starting grant no. 714034 SMART and ERC no. 850956}}
\date{}
\maketitle

\begin{abstract}
We prove new separability results about free groups. Namely, if $H_1, \ldots , H_k$ are infinite index, finitely generated subgroups of a non-abelian free group $F$, then there exists a homomorphism onto some alternating group $f:F \twoheadrightarrow A_m$ such that whenever $H_i$ is not conjugate into $H_j$, then $f(H_i)$ is not conjugate into $f(H_j)$. 

The proof is probabilistic. We count the expected number of fixed points of $f(H_i)$'s and their subgroups under a carefully constructed measure.
\end{abstract}

\section{Introduction}
Let's say that we want to understand a typical homomorphism between two groups. The simplest domain would be a free group because then the map is specified by its values on generators. The correspondence between the maps and the tuples is bijective, so studying maps from free groups is the same as studying tuples of elements. To make this interesting, we need a family of groups, ideally one which is easy to describe and to work with. In this paper, we take the symmetric groups.

We prove the following theorem and develop a technique to get a range of similar results.
\begin{letterthm}\label{LetterTheoremForSection3}
Suppose $H_1, \ldots , H_k$ are infinite index, finitely generated subgroups of a non-abelian free group $F$. Then there exists a surjective homomorphism $f:F \longrightarrow A_m$ such that whenever $H_i$ is not conjugate into $H_j$, then $f(H_i)$ is not conjugate into $f(H_j)$. Moreover, there exists a surjective homomorphism $f':F \longrightarrow S_{m'}$ with the same property.
\end{letterthm}

%\begin{proof}
%Suppose $f:F \longrightarrow S_m$. Let $H_{i,n}$ be the intersection of all index $n$ subgroups of $H_i$. If $f(H_i)$ is conjugate into $f(H_j)$ then $f(H_{i,n})$ is conjugate into $f(H_{j,n})$ and in particular $H_{i,n}$ fixes at least as many points as $H_{j,n}$. Conjugacy classes of $F$ form a poset, so after relabelling and removing duplicates, we can assume that if $H_i$ is conjugate into $H_j$, then $i \leq j$. Let $M$ be a (large) number to be specified later and $a_1, a_2, \ldots, a_n$ be a (steeply increasing) sequence also to be specified later. Let $G$ consist of $M$ copies of core graphs for each $H_{i,a_i}$. Complete $G$ randomly. For large enough $m$ the following happens with positive probability: Whenever $i<j$, $H_{i, a_i}$ fixes more points than $H_{j, a_i}$ and $Im(f) = A_m$. Also for large enough $m$ the following happens with positive probability: Whenever $i<j$, $H_{j, a_i}$ fixes fewer points than $H_{i, a_i}$ and $Im(f) = S_m$.
%
%But for the $f$ above $H_j$ is not conjugate into $H_i$.
%
%(Check correctness of the proof and cite theorems in the paper.)
%\end{proof}

This theorem gives an alternating version of \emph{conjugacy separability (SCS)} and \emph{subgroup into-conjugacy separability (SICS)} for free groups. These properties allow one respectively to preserve non-conjugacy of two subgroups under passing to a finite quotient and to preserve the property of one subgroup not being conjugate into another. These properties were conceived of by Bogopolski-Grunewald \cite{bogopolski2010subgroup}. In the same paper they proved these properties for free groups and free products of finite groups \cite{bogopolski2010subgroup}. Later Bogopolski-Bux proved these properties for surface groups and related them to curvature-type properties \cite{bogopolski2014subgroup}. Chagas-Zalesskii proved that free-by-finite groups are SCS and property SCS is preserved under free products \cite{chagas2015subgroup} and that limit groups are SCS \cite{chagas2016limit}.

Our theorem generalizes SICS for infinite subgroups to the setting of alternating or symmetric quotient. Indeed if $H_1$ is not conjugate into $H_2$, we can make $f(H_1)$ not conjugate into $f(H_2)$, where $f$ is onto an alternating (resp. symmetric) group. It also generalizes SCS. Indeed the relation `is conjugate into' is antisymmetric on the conjugacy classes of finitely generated subgroups of a free group \cite[Lemma 2.1]{bogopolski2010subgroup} and if finitely generated infinite index subgroups $H_1, H_2 <F$ are not conjugate, then without loss of generality $H_1$ is not conjugate into $H_2$ and we can make $f(H_1)$ not conjugate into $f(H_2)$ under some alternating (resp. symmetric) quotient and then $f(H_1)$ and $f(H_2)$ are not conjugate.
%
%\begin{lemma}
%If finitely generated $H_i <F$ is conjugate into $H_j <F$ and $H_j$ is conjugate into $H_i$, then $H_i$ and $H_j$ are conjugate.
%\end{lemma}
%
%\begin{proof}
%Assume that $H_i$ and $H_j$ are not conjugate. Then $H_i$ is conjugate to a proper subgroup of itself, say $g H_i g^{-1} \lneq H_i$. Let $Y$ be the cover associated to $H_i$. Then $H_i \gneq gH_i g^{-1} \gneq g^2 H_i g^{-2} \gneq \ldots$ Conjugating by $g$ extends to a self-map $f$ of $(Y,y_0)$. This map is label-preserving and hence uniquely determined by the image of $y_0$. It sends $y_0$ to the endpoint of path labeled $g^{-1}$ which starts at $y_0$. It sends loops to loops, so it cannot send $y_0$ too far from the core of $Y$. But $Y$ is locally finite and the core is finite and hence there exist $k>l >0$ such that $f^k y_0 = f^l y_0$. Therefore $g^k H_i g^{-k} = g^l H_i g^{-l}$, which contradicts $g^k H_i g^{-k} \lneq g^l H_i g^{-l}$.
%\end{proof}

Once we build the probabilistic machinery, we can also easily reprove the main theorem from \cite{wilton2012alternating}, which is the alternating analogue of subgroup separability in the free group. See section \ref{Wilton's theorem}.

As the size of a symmetric group increases, the probability that two random permutations generate the entire symmetric group or its index $2$ alternating subgroup tends to $1$. The proof of this fact is much harder than the trivial task of finding a pair of permutations, which generate the symmetric group. However, an explicit construction of quotients becomes trickier, more tedious or potentially unknown as the complexity of desired properties increases. The probabilistic method bypasses this by focusing on statistics of outcome rather than the individual cases.

Random actions of groups have been extensively studied before. For example, Liebeck-Shalev studied random quotients of Fuchsian groups \cite{liebeck2004fuchsian}, walks in finite groups of Lie type \cite{liebeck2005character}, spans of elements of fixed order in finite simple groups \cite{liebeck2002random}. However, our techniques are most similar to how Puder-Parzanchevski counted fixed points of a subgroup of a free group under a random permutation action \cite{puder2015measure}.
Probabilistic methods had been used before to prove that for every infinite class $\mathcal{C}$ of simple groups, every non-abelian free group is residually $\mathcal{C}$ \cite[Theorem 3]{dixon2003residual}. 

People had also asked before in various settings: ``What is a typical quotient?" One can take two random elements \cite{dixon1969probability} or even one restricted element and the other at random \cite{babai1989probability}. The results of this paper enable us to impose restrictions on both (or all) generators simultaneously.

In Section \ref{SectionSetUp} we say what we mean by these restrictions and state that under mild conditions most of quotients satisfying these restrictions are alternating or symmetric. The next few sections are devoted to the proof of this statement.
 In Section \ref{Transitivity}, we prove that the random group is transitive. In Section \ref{Primitivity} we prove that it is also primitive. In Section \ref{Jordan} we prove that the random group contains a short cycle and that this together with primitivity proves the theorem.
 
Section \ref{Wilton's theorem} illustrates how to use the results from Section \ref{SectionSetUp} by quickly reproving a known theorem.

In Section \ref{Applications}, we apply the theorem to show new separability properties of free groups. In particular, infinite index, finitely generated non-conjugate subgroups of a free group map to non-conjugate subgroups of an alternating group under some surjective homomorphism onto an alternating group.

\section{Set-up} \label{SectionSetUp}
A generalisation of the following theorem allows us to show that certain groups are alternating with a large probability.
\begin{theorem}\label{DixonTheorem} \cite{dixon1969probability}
An image of a random homomorphism $F_2 \longrightarrow S_n$ is $A_n$, resp. $S_n$, with probabilities which tend to $1/4$, resp. $3/4$, as $n$ goes to infinity.
\end{theorem}

We generalise this result to the setting with finitely many  conditions on the generators $a_1, \dots a_k$ of $F_k$. These conditions are given by an immersion of a finite graph into a rose via a correspondence which we now discuss.
The basic idea is to start with a graph, which extends to a covering of the presentation complex. We then look at all the ways it extends to a covering.

\begin{definition}[Schreier graphs]
If $G$ is a group and $\mathcal{H} = \{ H_i: i \in I \}$ a family of subgroups of $G$ and $\mathcal{A}$ is a collection of elements of $G$, then \emph{the Schreier graph associated to the subgroups $\mathcal{H}$ and elements $\mathcal{A}$} consists of:
\begin{itemize}
\item A vertex for every coset $gH_i$ for every $H_i \in \mathcal{H}$.
\item An edge labeled $a$ from the vertex corresponding to $gH_i$ to the vertex corresponding to $a gH_i$ for all $a \in \mathcal{A}$ and all cosets of $H_i$.
\end{itemize}
\end{definition}

In the usual definition of the Schreier graph, there is just one subgroup, but we want to allow for disconnected graphs. We will often omit to specify $\mathcal{H}$ and $\mathcal{A}$, when they're clear from the context.

If $a_1, \ldots , a_k$ is a $k$-tuple of elements in $S_n$, we can construct a graph encoding their action on $\{1, \ldots, n\}$ as follows. Take $n$ vertices labelled $1, \ldots , n$ with $i$ and $a_j(i)$ connected by an oriented edge labelled $a_j$ for all $i$ and $j$. This graph is a (not necessarily connected) covering of \emph{the rose of $k$ petals $R_k$}, a graph which has a single vertex and $k$ edges labelled $a_1, \ldots , a_k$ respectively. The covering has degree $n$. 

Choose an arbitrary basepoint $v$ in each component and take $H_v$ to be the group consisting of labels of closed paths starting at $v$. Then the above graph is in fact the Schreier graph of $F_k = \langle a_1, \ldots, a_k \rangle$ associated to subgroups $\{H_v : v \text{ is a basepoint } \}$ and elements $\{a_1, \ldots, a_k \}$.

This gives a bijective correspondence between the degree $n$ coverings of $R_k$ and $k$-tuples of elements of $S_n$. To see the other direction, we need the following definition.

\begin{definition}[Core graph] \label{CoreDefinition}
Given a graph $Y$, \emph{the core of $Y$}, denoted $\operatorname{Core}(Y)$ is the subgraph of $Y$, which consists of all vertices and edges that are contained in some cycle.

Given $\Gamma$  a subgroup of a free group $F_k = \langle a_1, \ldots, a_k \rangle$, let $X_{\Gamma}$ be the Schreier graph associated to the subgroup $\Gamma$ and elements $\langle a_1, \ldots, a_k \rangle$.   \emph{The core of $\Gamma$}, denoted $\operatorname{Core}(\Gamma)$ is the graph $\operatorname{Core}(X_\Gamma)$.
\end{definition}
Note that $X_{\Gamma}$ is a cover of $R_k$.

We will also need to talk about subspaces of covers without explicitly referring to the fact that it's a subspace of a cover.
%\begin{definition}[Precover]
%An oriented graph with edges labeled from among $a_1, \ldots, a_k$ is called a \emph{precover} if no vertex has no more than one incoming edge of every label and no more than one outgoing edge of every label.
%\begin{definition}
%To get a precover from a $k$-tuple take the core of the covering space associated to the subgroup generated by the elements in the $k$-tuple.

We will in general look at the coverings where the vertices are not labelled. This means that in fact we'll be using the correspondence between unlabelled degree $n$ coverings and conjugacy classes of $k$-tuples of elements of $S_n$. We can use this correspondence to define conditions on a random homomorphism from a finitely generated free group to the symmetric group $S_n$ as follows.

%Consider a labelled graph $G$ equipped with a locally injective label preserving map to $R_2$. We can extend $G$ to a covering by "adding edges whenever we can". More precisely let $A^+$ be the set of vertices of $G$ without an outgoing edge labelled $a$. Let $A^-$ be the set of vertices without an incoming edge labelled $a$. Pick a bijection $\alpha$ between $A^+$ and $A^-$ and add an oriented edge labelled $a$ from $v$ to $f(v)$ for all $v \in A^+$. Define $B^+, B^-$ and $\beta$ analogously for $b$. The resulting graph $H$ is a covering of $R_2$.

\begin{definition}[Precover, random action]
Suppose that $G$ is a labeled oriented graph and $G \longrightarrow R_k$ sends vertices to vertices and edges to the edges of the same label. Such a map is called a \emph{precover} of $R_k$. Just as a degree $n$ cover corresponds to a permutation $f: [n] \longrightarrow [n]$, a degree $n$ precover corresponds to a partial injective function $f: [n] \pfun [n]$.

Suppose $G$ has at most $n$ vertices. Add vertices to $G$ until there are $n$ vertices in total: let $G'$ be disjoint union of $G$ and a discrete graph with $n - |G|$ vertices.

Let $V_j^{no}(G')$ be the set of vertices of $G'$ without an outgoing edge labelled $a_j$ and $V_j^{ni}(G')$ be the set of the vertices without an incoming edge labelled $a_j$. For all $j$, choose a bijection $f_j$ between $V_j^{no}(G')$ and $V_j^{ni}(G')$ uniformly at random. Connect $v$ and $f_j(v)$ by an oriented edge labelled $a_j$.

The resulting graph $\overline{G}$ is \emph{a random degree $n$ completion} of $G$, the associated homomorphism $\varphi:F_k \longrightarrow S_n$ is \emph{a random homomorphism with condition $G$} and the associated group $\Gamma_n(G) \leqslant S_n$ is \emph{a random group with condition $G$}. Let's call $G$ a \emph{condition graph}.

%If we picked $\alpha$ and $\beta$ above uniformly at random, we call $H$ \emph{a random completion $X(G)$ of $G$}.
%The corresponding subgroup $\Gamma(G)$ of $S_{|G|}$ is \emph{a random subgroup associated to $G$}.
\end{definition}

We frequently take the condition graph to be a core graph, union of core graphs or some slightly larger superspace of a core graph. A core graph of a finitely generated group is a finite graph, since it is a union of only finitely many cycles. Recall from Theorem \ref{DixonTheorem} that $\Gamma_n(\emptyset)$ is frequently $S_n$ or $A_n$.
If some component of a graph $G$ is an actual covering of $R_2$, then $\Gamma_n(G)$ is non-transitive for $n > |G|$. We prove a converse result:
\begin{theorem}[Main Theorem]\label{MainTheorem}
If no component of $G$ is a covering of $R_k$, then $\Gamma_n(G)$ is $S_n$ or $A_n$ with probabilities which tend to $1-2^{-k}$ or $2^{-k}$ respectively as $n$ goes to infinity.
\end{theorem}

\section{Transitivity} \label{Transitivity}
We need to show that a random group is either $S_n$ or $A_n$. Both $A_n$ and $S_n$ are transitive, so the transitivity is necessary. It also turns out to be one of the conditions used in the converse statement.

\begin{lemma} [\cite{dixon1969probability}]\label{Dixon Theorem}
The group $\Gamma_n(\emptyset)$ is almost always transitive.
(i.e. the probability that $\Gamma_n(\emptyset)$ generates a transitive subgroup of $S_n$ tends to $1$ as $n$ goes to infinity).
\end{lemma}

If a component of $G$ is an actual covering, then no completion is transitive (except for the case when the component is all of $G$ and there are no other vertices). That component remains a component in any completion. We need to exclude this situation in the generalised version of the theorem.

\begin{lemma}
Assume that no component of a graph $G$ is a covering of $R_k$.
Then the group $\Gamma_n(G)$ is almost always transitive and a random completion $\overline{G}$ is almost always connected.
\end{lemma}

The idea of the proof is as follows. We're starting from something which intuitively is more connected than a discrete graph. We formalise this intuition by constructing a probability preserving map between random completions of $\emptyset$ and random completions of $G$, which preserves connectedness. We will do this by replacing components of $G$ with discrete graphs.

\begin{proof}

Let $G_1, G_2, \dots, G_l$ be the connected components of $G$. Let $E_j (G_i)$ be the set of edges labelled $a_j$ in $G_i$.
\begin{itemize}
\item Case 1:
The number of edges $|E_j(G_i)|$ labelled $a_j$ in $G_i$ is the same for all $j$. Let $H_i$ be the discrete graph with $|V(G_i)| - |E(G_i)|/l$ vertices. Let $H$ be the union of all $H_i$. Pick a bijection between the "missing edges" at vertices of $H_i$ and the "missing edges" at vertices of $G_i$ - see figure \ref{transitivityFigureExample}. This induces a map between random completions. More formally, recall that if $G$ is a graph, then $V_j^{ni}(G)$ and $V_j^{no}(G)$ are the vertices with no incoming and no outgoing edge labelled $a_j$, respectively. The label $a_j$ appears the same number of times in $G_i$ for all $j$, so  
$$|V_j^{ni}(G_i)| = |V_j^{no}(G_i)| = |V(G_i)|- |E_j(G_i)|=|V(G_i)| - |E(G_i)|/l$$
 is independent of $j$, where $E_j(G_i)$ are the edges of $G_i$ with label $a_j$. The graph $H_i$ is discrete, so we have $|V_j^{ni}(H_i)|=|V(G_i)| - |E(G_i)|/l$. Pick arbitrary bijections $f_{i,j}^{ni}:V_j^{ni}(H_i) \longrightarrow V_j^{ni}(G_i) $. Let $f$ be a union of these bijections. These maps induce a bijection between the degree $n$ completions of $H$ and degree $n+ |E_j(G_i)|$ completions of $G$ as follows. Given a completion $\overline{H}$ of $H$, consider $(\overline{H} \setminus H) \cup G$. Now connect each open end of an edge in $(\overline{H} \setminus H)$, which was previously attached to $v \in H$ to $f(v)$. This is a completion of $G$. Call it $f(\overline{H})$ by abuse of notation. This correspondence is bijective as now we could excise $G$ and connect the open ends back to $H$. 
 
Suppose $f(\overline{H}) = K_1 \sqcup K_2$, where $K_i$ is closed non-empty. For all $v \in H$, the component of $G$ containing $f_{i,j}^{no}(v)$ and $f_{i,j}^{ni}(v)$ does not depend on $j$. Hence, the closures of $K_i \setminus (G \cap K_i)$ in $\overline{H}$ are two disjoint closed subsets partitioning $\overline{H}$. They are non-empty as long as $K \not\subset G$. This is where we use that no component of $G$ is a cover. If $\overline{H}$ is connected, then so is $f(\overline{H})$.
  
The probability that a random completion of $H$ is connected (hence the associated group is transitive) tends to $1$ by Lemma \ref{Dixon Theorem} and therefore the probability that a random completion of $G$ is connected also tends to $1$.

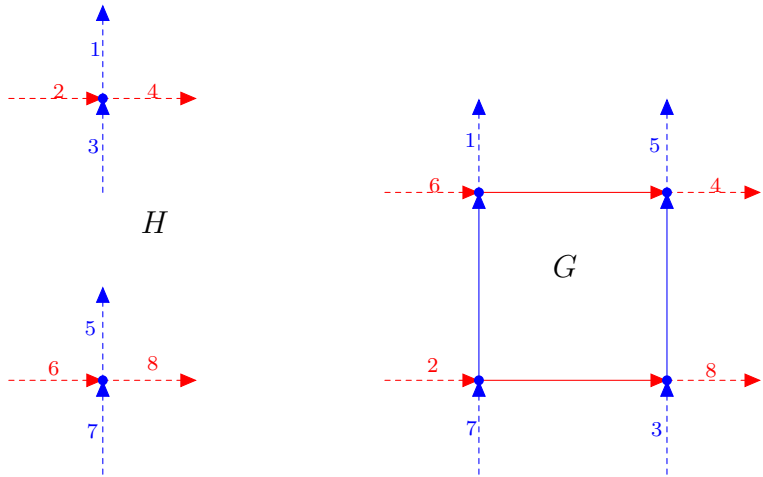
\begin{figure}
\centering
\definecolor{ffqqqq}{rgb}{1,0,0}
\definecolor{qqqqff}{rgb}{0,0,1}
\begin{tikzpicture}[line cap=round,line join=round,>=triangle 45,x=1.0cm,y=1.0cm, scale=1.25]
\clip(-1.29,-2.67) rectangle (7.57,3.34);
\draw [->,dash pattern=on 2pt off 2pt,color=ffqqqq] (0,-1) -- (1,-1);
\draw [->,dash pattern=on 2pt off 2pt,color=ffqqqq] (-1,-1) -- (0,-1);
\draw [->,dash pattern=on 2pt off 2pt,color=qqqqff] (0,-2) -- (0,-1);
\draw [->,dash pattern=on 2pt off 2pt,color=qqqqff] (0,-1) -- (0,0);
\draw [->,dash pattern=on 2pt off 2pt,color=ffqqqq] (0,2) -- (1,2);
\draw [->,dash pattern=on 2pt off 2pt,color=ffqqqq] (-1,2) -- (0,2);
\draw [->,dash pattern=on 2pt off 2pt,color=qqqqff] (0,1) -- (0,2);
\draw [->,dash pattern=on 2pt off 2pt,color=qqqqff] (0,2) -- (0,3);
\draw [->,dash pattern=on 2pt off 2pt,color=ffqqqq] (3,-1) -- (4,-1);
\draw [->,dash pattern=on 2pt off 2pt,color=qqqqff] (4,-2) -- (4,-1);
\draw [->,dash pattern=on 2pt off 2pt,color=qqqqff] (4,1) -- (4,2);
\draw [->,dash pattern=on 2pt off 2pt,color=ffqqqq] (3,1) -- (4,1);
\draw [->,dash pattern=on 2pt off 2pt,color=ffqqqq] (6,1) -- (7,1);
\draw [->,dash pattern=on 2pt off 2pt,color=qqqqff] (6,1) -- (6,2);
\draw [->,dash pattern=on 2pt off 2pt,color=ffqqqq] (6,-1) -- (7,-1);
\draw [->,dash pattern=on 2pt off 2pt,color=qqqqff] (6,-2) -- (6,-1);
\draw [->,color=qqqqff] (4,-1) -- (4,1);
\draw [->,color=ffqqqq] (4,-1) -- (6,-1);
\draw [->,color=qqqqff] (6,-1) -- (6,1);
\draw [->,color=ffqqqq] (4,1) -- (6,1);
\draw (0.29,0.91) node[anchor=north west] {$H$};
\draw (4.67,0.45) node[anchor=north west] {$G$};
\begin{scriptsize}
\fill [color=qqqqff] (0,-1) circle (1.5pt);
\draw[color=ffqqqq] (0.53,-0.82) node {8};
\draw[color=ffqqqq] (-0.52,-0.87) node {6};
\draw[color=qqqqff] (-0.11,-1.54) node {7};
\draw[color=qqqqff] (-0.13,-0.45) node {5};
\fill [color=qqqqff] (0,2) circle (1.5pt);
\draw[color=ffqqqq] (0.53,2.08) node {4};
\draw[color=ffqqqq] (-0.47,2.08) node {2};
\draw[color=qqqqff] (-0.1,1.49) node {3};
\draw[color=qqqqff] (-0.08,2.52) node {1};
\fill [color=qqqqff] (4,-1) circle (1.5pt);
\fill [color=qqqqff] (6,-1) circle (1.5pt);
\fill [color=qqqqff] (4,1) circle (1.5pt);
\fill [color=qqqqff] (6,1) circle (1.5pt);
\draw[color=ffqqqq] (3.51,-0.84) node {2};
\draw[color=qqqqff] (3.92,-1.51) node {7};
\draw[color=qqqqff] (3.91,1.56) node {1};
\draw[color=ffqqqq] (3.53,1.08) node {6};
\draw[color=ffqqqq] (6.52,1.08) node {4};
\draw[color=qqqqff] (5.87,1.5) node {5};
\draw[color=ffqqqq] (6.47,-0.89) node {8};
\draw[color=qqqqff] (5.89,-1.52) node {3};
\end{scriptsize}
\end{tikzpicture}
\caption{The graph $G$ is the core of $\langle [a,b] \rangle$ and $H$ consists of two vertices. Pick a bijection between the missing edges at $H$ and the missing edges at $G$. A completion of $H$ corresponds to a completion of $G$ by reconnecting the adjacent edges according to this bijection. If the completion of $H$ is connected, then so is the completion of $G$.}\label{transitivityFigureExample}
\end{figure}

\item
Case 2:
Suppose $|E_j(G_i)|$ is not independent of $j$. We can reduce this situation to case 1, by taking a slightly larger graph $G'$, which satisfies this condition. The key observation will be that most completions of $G$ are also completions of $G'$. 

If there is some $i, j$ and $j'$ with $|E_j(G_i)| < |E_{j'}(G_i)|$, let $v_j$ be a vertex of $G_i$ with no outgoing edge labelled $a_j$. Replace $G_i$ by a union of $G_i$ and an $a_j$-edge starting at $v_j$ and ending at a new leaf. Repeat this process until $|E_j(G_i)|$ becomes independent of $j$.

This process terminates since $\sum_i \sum_j \left( \max_{j'}\left(|E_{j'}(G_i)|\right) - |E_{j}(G_i)|\right)$  is a non-negative integer, which decreases whenever we change the graph.
Let $G'$ be the resulting graph.

The inclusion of $G$ to $G'$ is a $\pi_1$-isomorphism on each component and $G'$ contains finitely many more edges than $G$. If $\overline{G}$ is a random completion of $G$, then there is a unique map $G' \longrightarrow \overline{G}$ extending the inclusion of $G$. If this map is injective, then $\overline{G}$ is also a completion of $G'$. Let's estimate the probability of this event. Build a random completion of $G$ in the same way, we've built $G'$: one edge at a time.

If first edge $e_1$ connects to a vertex of $G$, then the injectivity fails. There are $n-|V(G)|$ vertices not in $G$. If $e_1$ connects to one of them, we can continue with the second edge. The second edge $e_2$ can fail the injectivity in at most $|V(G)|+1$ ways (it might connect back to $G$ or to an endpoint of $e_2$). It can succeed in at least $n-|V(G)|-1$ ways. Continue for all new edges. The probability that $G' \longrightarrow \overline{G}$ is injective is at least 
$$\frac{n-|V(G)|}{n}\frac{n-|V(G)|-1}{n} \ldots \frac{n-|V(G)|-\Delta}{n}$$
where $\Delta = |E(G')|-|E(G)|$. This quantity goes to $1$ as $n$ goes to infinity. This means $G' \longrightarrow \overline{G}$ is almost always injective and a completion of $G$ is almost always a completion of $G'$. By case 1, a completion of $G'$ is almost always connected, therefore a completion of $G$ is almost always connected. We are implicitly using that the probabilities are compatible in the following sense.

\begin{dmath*}
\mathbb{P}(\text{ A completion of } G \text{ is } H | H  \text{ is a completion of } G' ) =
 \mathbb{P}(\text{ A completion of } G' \text{ is } H)
 \end{dmath*}

This is true, because it does not matter whether we complete $G$ to a completion containing $G'$, or whether we complete $G'$.

\end{itemize}

\end{proof}

\section{Primitivity} \label{Primitivity}
An action of a group $\Gamma$ on a finite set $X$ is \emph{primitive} if it is transitive and no nontrivial partition of $X$ is preserved by $\Gamma$. We have already dealt with the transitivity, so we just need to show non-existence of a preserved partition. Transitivity implies that all sets in the partition have the same size, hence taking $n$ to be a prime ensures primitivity, but we do not need to do that here.

\begin{lemma}
Assume that no component of $G$ is a covering of $R_k$.
Then $\Gamma_n(G)$ is almost always primitive.
\end{lemma}

We use that imprimitive groups are extremely rare.

\begin{proof}
By Lemma 2 in \cite{dixon1969probability}, the proportion of pairs of elements of $S_n$, which generate an imprimitive subgroup is at most $n2^{-\frac{n}{4}}$ (and hence this bound also applies to $k$-tuples).

Let's count what proportion of $k$-tuples of elements of $S_n$ \emph{respects} $G$ (i.e. how many arise from a completion of $G$).

Recall that $|E_j(G)|$ is the number of edges in $G$ labelled $a_j$.

The probability that a random permutation moves vertices according to the edges labelled $a_j$ is $$\frac{1}{n(n-1) \ldots (n- |E_j(G)| +1)}\,.$$

If $n > 2 |E(G)|$, a random completion respects $G$ with probability at least $(2n)^{-|E(G)|}$. This is only polynomial in $n$. Even if all $k$-tuples generating imprimitive subgroups respected $G$, the proportion of imprimitive random completions of $G$ would be at most
$$ \frac{n2^{-\frac{n}{4}}}{(2n)^{-|E(G)|}}=(2n)^{|E(G)|}n2^{-\frac{n}{4}}$$
which goes to zero as $n$ goes to $\infty$.
\end{proof}
\section{`Jordan' condition} \label{Jordan}
The final condition (in addition to being primitive) for a subgroup to be $A_n$ or $S_n$ is that it contains a $q$-cycle for some prime $q \leq n-3$ \cite[Theorem 13.9]{wielandt2014finite}.

Following \cite{dixon1969probability}, we define $C_{q,n} \subset	S_n$ to consist of those permutations which contain a single cycle of length divisible by $q$ and all the other cycles are of lengths coprime to $q$.

In particular, if $G$ contains an element of $C_{q,n}$, then it contains a $q$-cycle. The following lemma is a key step in Dixon's theorem.

\begin{lemma}[Lemma 3 in \cite{dixon1969probability}]\label{primeEstimate}
Let $T_n = \bigcup_q C_{q,n}$, where the union is over all primes $q$ such that
$$(\log{n})^2 \leq q \leq n-3\,.$$
Then the proportion $u_n$ of elements of $S_n$ which lie in $T_n$ is at least
$$ 1 -4/(3 \log \log n)$$
for all sufficiently large $n$.
\end{lemma}

We need to generalise this to the conditional case.
\begin{lemma}
Let $G$ be any graph. Take a random group action with condition $G$.
Almost always some power of $a_1$ acts as a $q$-cycle, where $q \leq n-3$ is a prime.
\end{lemma}

The generalisation is a bit more complicated. We separate the $a_1$-edges in the condition graph $G$ into cycles and paths. We will take $n$ very large compared to the size of the cycles. This will allow us to ignore the cycles since they will all be smaller than the prime $q$. To deal with the paths, one only needs to realise that paths are a typical behaviour. The corresponding walks in the random unconditional completion would almost always be injective, so we can apply the unconditional theorem.

\begin{proof}
We are only using one generator, so in this proof we can assume that there is only one generator.
The condition graph $G$ consists of cycles and paths because no vertex has valence greater than $2$. We will deal with both of them separately. The paths do not really cause many issues. As in the proof of transitivity, almost every completion of an empty graph will be also a completion of a union of paths. This will reduce the statement to the unconditional version. To deal with the cycles we can use the lower bound of Lemma \ref{primeEstimate} and force $q$ to be bigger than the length of all cycles. This way a suitable power of $a_1$ will fix the cycles pointwise, and act as a $q$-cycle on the remaining vertices.

The graph $G$ consists of paths $P_1, \ldots , P_k$ and cycles $L_1, \ldots , L_l$. Let $v_i$ be the initial vertex of $P_i$.
Let $G'$ be the union of all the paths $P_i$.

Let $n' = n - \sum_i |L_i|$. Let $D_k$ be a graph with $k$ vertices and no edges. Pick a bijection $f$ between the vertices of $D_k$ and $\{ v_i\}$. Consider the random degree $n'$ completion $\Gamma_{n'}(D_k)$. Then by lemma \ref{primeEstimate}
$$\mathbb{P}(a_1 \text{ acts as an element of } T_{n'}) \geq 1 - 4/(3 \log \log n')$$
for sufficiently large $n'$.

There is unique label and orientation preserving map $\overline{f}$ from $G'$ to a completion of a discrete graph, which extends $f$. If this map $\overline{f}$ is injective, then the completion of $D_k$ is also a completion of $G'$. We claim that this happens with probability $1-\mathcal{O}(1/n')$. Let's proceed by induction on the sum of length of the paths in $G'$. If there are no edges, the map $\overline{f}$ is just $f$ and therefore a bijection to its image $D_k$.

If $G'$ contains an edge, let $e$ be an edge at the end of one of the paths. Let $G''$ be $G'$ without $e$ and the terminal endpoint $t(e)$, but with the initial endpoint $i(e)$. In other words, $G''$ is the same graph as $G'$, just with one of the paths shorter by $1$. By induction $G''$ injects with probability $1-\mathcal{O}(1/n')$. Suppose $G''$ injects. Then the graph $G'$ fails to inject only if $t(e)$ is one of the vertices in $D_k$. This happens with probability $ \frac{k}{n' - |E(G'')|}$ since there are $n'- |E(G'')| - k$ vertices not in the image of $G'$. Therefore, $G'$ injects with probability $\left( 1-\mathcal{O}(\left( 1/n '\right) \right) \left( 1-  \frac{k}{n' - |E(G'')|} \right) = \left( 1-\mathcal{O}(1/n') \right)$.

A random completion of $D_k$ is almost always a random completion of $G'$. We can restate Lemma \ref{primeEstimate} as follows. A random completion of $D_k$ has almost always the property that the induced $a_1$ belongs to $T_{n'}$. But then the same applies to a random completion of $G'$, because a random completion of $D_k$ is almost always a completion of $G'$. I.e. some power of $a_1$ in the random action with condition $G'$ almost always acts as $q$-cycle, where $q$ is a prime with $(\log n)^2 \leq q \leq n-3$.

Take $n' > \exp ( \sqrt{\max |L_i|})$. A random completion of $G$ is just a union of a random completion of $G'$ and the cycles $L_i$. Therefore $a_1$ almost always acts as a union of an element from $T_{n'}$ and cycles of lengths $|L_i|$. By choice of $n'$, we have $\max |L_i| < q$. Some power of $a_1$ almost always acts as a union of a $q$-cycle and cycles shorter than $q$. Therefore, a higher power of $a_1$ almost always acts as $q$-cycle.
\end{proof}

\section{Sample application - reproving alternating version of subgroup separability of free groups} \label{Wilton's theorem}
We can now reprove the main theorem of \cite{wilton2012alternating} using probabilistic methods.

\begin{proposition}[\cite{wilton2012alternating}]
Suppose $G$ is a finitely generated infinite index subgroup of a non-abelian free group $F_k$ and that $g_1, \ldots , g_l \in F_k \setminus G$. Then there exists a surjection $f:F_k \longrightarrow A_n$ onto some alternating group such that $f(g_i) \notin f(G)$ for all $i$.
\end{proposition}

The technique is similar to and illustrative of the proof of theorem \ref{LetterTheoremForSection3}.
\begin{proof}
Let $X_G$ be the cover of $R_k$ associated to $G$ as before and let $x_g$ be a basepoint in $X_G$ such that the closed paths starting at $x_g$ are precisely the elements of $G$. Let $\gamma_i$ be the loop in $R_k$ representing $g_i$ and let $\tilde{\gamma_i}$ be its lift to $X_G$ starting at $x_G$. Let $Y \subset X_G$ be the union of all loops in $X_G$ and all images of $\tilde{\gamma_i}$.

The graph $Y$ consists of a single component and this component is not a covering of $R_k$ as  $X_G$ is a connected infinite degree cover and its finite (non-empty) subgraphs are not coverings. We can apply Theorem \ref{MainTheorem} which says that a random group $\Gamma_n(Y)$ with condition $Y$ is $S_n$ or $A_n$ with probabilities which tend to $1-2^{-k}$ and $2^{-k}$ respectively as $n$ goes to infinity. In particular, probability that the image of $G$ is $A_n$ is eventually positive. The endpoint of $\tilde{\gamma_i}$ isn't $x_G$, therefore $f(g_i) \notin f(G)$ in any completion of $Y$ and in particular also in those which surject onto an alternating group. 
\end{proof}

\section{Subgroup conjugacy separability and randomness} \label{Applications}
In this section we prove Theorem \ref{LetterTheoremForSection3}.
A random action often demonstrates separability properties of a free group. Since the action is often alternating, this demonstrates separability within alternating groups.

Let $g$ and $h$ be two elements of a free group, such that $g$ is not conjugate to either $h$ or $h^{-1}$. After conjugation, we may assume that $g$ is cyclically reduced and freely reduced. If a homomorphism $f: F_2 \longrightarrow S_n$ is such that $f(g)$ and $f(h)$ have different cycle structures, then $g$ and $h$ remain in different conjugacy classes in the image under $f$.

A random action with a suitable condition will give different expected numbers of fixed points of $g$ and $h$ and just a small variance. This produces actions, which keep $g$ and $h$ in different conjugacy classes.

Let $G$ be a loop labelled with $g$. In counting fixed points of $g$, we need to count how often $G$ lifts to a covering. We can categorise these lifts by their image. I.e. we can count injective lifts of possible images of $G$.

%\begin{definition}[Fold]
%Suppose $(G,u)$ and $(H,v)$ are two pointed graphs. \emph{The fold of $(G,u)$ and $(H,v)$} is $(G,u) \bigsqcup	(H,v) / \sim$, where $\sim$ is the smallest equivalence relation such that $u \sim v$, which makes the fold into precover. 
%\end{definition}
%Equivalently, we get a fold of $(G,u)$ and $(H,v)$ by identifying $u$ and $v$ and then iteratively folding edges, which violate local injectivity.

\begin{definition}[Quotient of a precover]
If $G$ is a precover of $R_k$ for some $k$ and $K$ is a graph, then a simplicial surjective locally injective map $f: G \twoheadrightarrow K$ is \emph{a quotient of a precover $G$}.
\end{definition}

Let's say we want to count the number of lifts of a graph $H$. Then the image of a lift of $H$ is some quotient of $H$. If we take the union with $G$, we get some quotient of $G \sqcup H$, where the restriction to $G$ is injective. Counting the lifts of $H$ is therefore the same as counting the injective lifts of those quotients of $G \sqcup H$, where the restriction to $G$ is injective. Let's give this quantity a notation.

\begin{definition}
Suppose $G$ and $H$ are precovers, $K$ is a quotient of the precover $G \sqcup H$ and $\overline{G}$ is a completion of $G$. By $\mu_{K \longrightarrow \overline{G}}$ we denote the number of injective maps from $K$ to $\overline{G}$ such the the composition $G \longrightarrow K \longrightarrow \overline{G}$ is the natural inclusion of $G$ to $\overline{G}$.

Let $\tau_{H \longrightarrow \overline{G}}$ be the total number of maps from $H$ to $G$.
\end{definition}

Note that if $G \longrightarrow K$ is not injective, then $G \longrightarrow K \longrightarrow \overline{G}$ cannot be an inclusion of $G$ and therefore $\mu_{K \longrightarrow \overline{G}} = 0$. Let's express $\tau$ using $\mu$.

\begin{lemma} \label{TotalToInjective}
Suppose that $G$ and $H$ are precovers and $\overline{G}$ is a completion of $G$. The total number of maps from $H$ to $\overline{G}$ is given by the following.
$$ \tau_{H \longrightarrow \overline{G}}=  \sum_{K = (G \sqcup H) / \sim} \mu_{K \longrightarrow \overline{G}}$$
The sum goes over quotients $K$ of the precover $G \sqcup H$. 
\end{lemma}

\begin{proof}
Given a map $f:H \longrightarrow \overline{G}$, let $K = G \cup f(H)$. Then $K$ injects to $\overline{G}$ and $G \longrightarrow K \longrightarrow \overline{G}$ is an isomorphism onto $G \subset \overline{G}$.

Conversely, if $K$ is a quotient of $G \sqcup H$ and it injects to $\overline{G}$ and $G \longrightarrow K \longrightarrow G$ is an isomorphism, let $f$ be the map $H \longrightarrow K \longrightarrow G$.
\end{proof}

We will now need to estimate each summand in the previous lemma. If $G$ was empty then the first order estimate would be $n^{\chi(K)}$ \cite[Theorem 1.8]{puder2015measure}. To take potentially non-empty $G$ into account, we define the relative Euler characteristic be a difference of the Euler characteristics.

\begin{definition}[Relative Euler Characteristic]
If $K$ is a quotient of $G \sqcup H$ such that $G$ embeds to $K$, then \emph{the Euler characteristic of $K$ relative to $G$} is $\chi_G (K) = \chi(K) - \chi(G)$.
\end{definition}

The next lemma gives the expected number of lifts of a quotient of $G \sqcup H$. This quantity makes intuitive sense, since the relative Euler characteristic counts the components of $K$ disjoint from components of $G$, minus the loops of $K$, which are not loops of $G$. See Figure \ref{LiftsExample}.

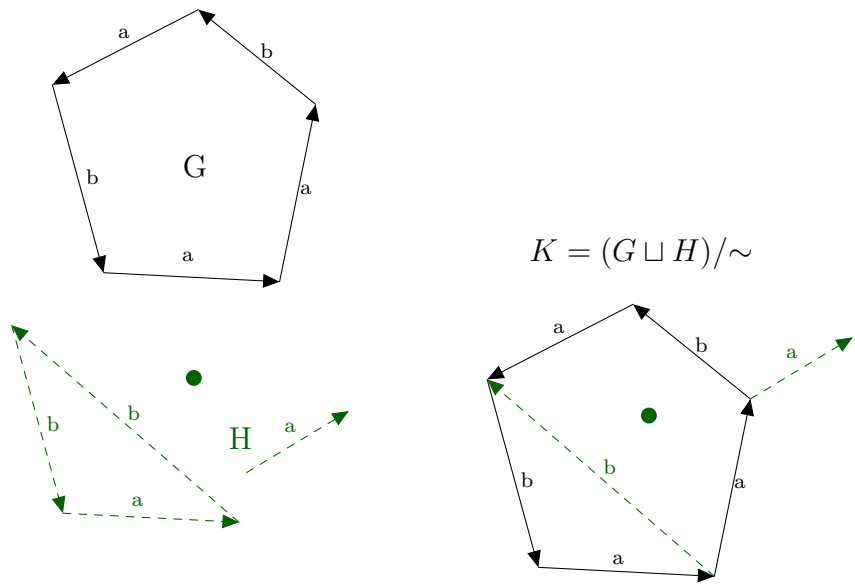
\begin{figure} 
\definecolor{qqwuqq}{rgb}{0,0.39,0}
\begin{tikzpicture}[line cap=round,line join=round,>=triangle 45,x=1.0cm,y=1.0cm]
\clip(-6.1,-0.99) rectangle (6.69,7.7);
\draw [->] (1.66,0.12) -- (4,0);
\draw [->] (4,0) -- (4.48,2.36);
\draw [->] (4.48,2.36) -- (2.92,3.62);
\draw [->] (2.92,3.62) -- (0.98,2.62);
\draw [->] (0.98,2.62) -- (1.66,0.12);
\draw [->,dash pattern=on 3pt off 3pt,color=qqwuqq] (4,0) -- (0.98,2.62);
\draw [->,dash pattern=on 3pt off 3pt,color=qqwuqq] (4.48,2.36) -- (5.84,3.18);
\draw [->,dash pattern=on 3pt off 3pt,color=qqwuqq] (-4.66,0.84) -- (-2.32,0.72);
\draw [->,dash pattern=on 3pt off 3pt,color=qqwuqq] (-5.34,3.34) -- (-4.66,0.84);
\draw [->,dash pattern=on 3pt off 3pt,color=qqwuqq] (-2.32,0.72) -- (-5.34,3.34);
\draw [->,dash pattern=on 3pt off 3pt,color=qqwuqq] (-2.22,1.38) -- (-0.86,2.2);
\draw [->] (-4.12,4.04) -- (-1.78,3.92);
\draw [->] (-1.78,3.92) -- (-1.3,6.28);
\draw [->] (-1.3,6.28) -- (-2.86,7.54);
\draw [->] (-2.86,7.54) -- (-4.8,6.54);
\draw [->] (-4.8,6.54) -- (-4.12,4.04);
\draw (-3.21,5.74) node[anchor=north west] {G};
\draw [color=qqwuqq](-2.6,2.13) node[anchor=north west] {H};
\draw (1.4,4.66) node[anchor=north west] {$K=(G \sqcup H)/ \sim$};
\begin{scriptsize}
\draw[color=black] (2.72,0.23) node {a};
\draw[color=black] (4.34,1.24) node {a};
\draw[color=black] (3.83,3.08) node {b};
\draw[color=black] (1.93,3.3) node {a};
\draw[color=black] (1.51,1.28) node {b};
\draw[color=qqwuqq] (2.61,1.46) node {b};
\draw[color=qqwuqq] (5.03,2.96) node {a};
\draw[color=qqwuqq] (-3.67,1) node {a};
\draw[color=qqwuqq] (-4.79,2.03) node {b};
\draw[color=qqwuqq] (-3.72,2.18) node {b};
\draw[color=qqwuqq] (-1.64,1.98) node {a};
\fill [color=qqwuqq] (-2.92,2.64) circle (3.0pt);
\draw[color=black] (-3,4.25) node {a};
\draw[color=black] (-1.43,5.15) node {a};
\draw[color=black] (-1.95,7) node {b};
\draw[color=black] (-3.85,7.22) node {a};
\draw[color=black] (-4.27,5.32) node {b};
\fill [color=qqwuqq] (3.13,2.14) circle (3.0pt);
\end{scriptsize}
\end{tikzpicture}
\caption{We expect roughly $1$ lift of $K$, since there is about $1/n$ probability that the diagonal $b$-edge closes up and there are about $n$ possibilities for the location of the isolated vertex. The green $a$-edge does not contribute anything, because there are roughly $n$ options for its endpoint and each of them appears with probability roughly $1/n$.}
\label{LiftsExample}
\end{figure}

\begin{lemma} \label{Estimate}
Suppose $G$ and $H$ are precovers and $K$ is a quotient of the precover $G \sqcup H$. Then we can express the expected number of maps from $K$ to the random completion $\overline{G}$ which extend the inclusion of $G$ as follows.
$$\E (\mu_{K \longrightarrow \overline{G}}) = n^{\chi_G (K)} + \mathcal{O}(n^{\chi_G (K)-1})$$
Here we fix $K, G$ and $H$ and we let $\overline{G}$ be a random degree $n$ completion of $G$.
\end{lemma}

\begin{proof}
We'll prove this by induction on the number of cells in $K \setminus G$. For the base case of $K = G$, left hand side is $1$ and the right hand side is $1 + \mathcal{O}(n^{-1})$.

\begin{enumerate}
\item Suppose there exists an edge $e$ of $K$ not contained in $G$. Let $K'$ be $K \setminus e$. By the induction on the number of cells, we have $\E (\mu_{K' \longrightarrow \overline{G}}) = n^{\chi_G (K')} + \mathcal{O}(n^{\chi_G (K')-1})$.

There are between $n$ and $n - |E(K')|$ ways for $e$ to lift and only one of them allows $K$ to lift. Hence,
$$\E (\mu_{K \longrightarrow \overline{G}}) = n^{-1} \E (\mu_{K' \longrightarrow \overline{G}}) + \mathcal{O}(n^{-2})$$

\item If $K \setminus G$ contains no edges, then it is a disjoint union of $G$ and vertices. Suppose $v \in K \setminus G$ is a vertex. Let $K' = K \setminus v$. To lift $K$, we need to lift $K'$ and specify, where does $v$ go. We always have between $n$ and $n - v(K')$ options for $v$, so 
$$\E (\mu_{K \longrightarrow \overline{G}}) = n \E (\mu_{K' \longrightarrow \overline{G}}) + \mathcal{O}(1)$$
\end{enumerate}
\end{proof}

In particular, we can get the highest order term approximation to the total number of expected lifts of $H$ to a completion of $G$ by determining the largest relative Euler characteristic among the quotients of $G \sqcup H$ and the number of quotients, which achieve this minimum.

\begin{definition}[Relative rank, critical graphs and multiplicity]
\emph{The relative rank} $r_G (H)$ is $\min \chi_G (K)$, where the minimum goes over quotients of $G \sqcup H$.

We call the quotients which achieve the minimum \emph{critical graphs}. \emph{Relative multiplicity} is the number of critical graphs.

\end{definition}

\begin{lemma} \label{varianceOfLifts}
Suppose $G$ and $H$ are precovers, and $G'$ a random completion of $G$. Then the variance of $\tau_{H \longrightarrow \overline{G}}$ is as follows\,.$$\operatorname{Var}(\tau_{H \longrightarrow \overline{G}}) = \E (\tau_{(H \sqcup H) \longrightarrow \overline{G}}) - \E(\tau_{H \longrightarrow \overline{G}})^2$$
%= \sum \E \mu_{K \longrightarrow \overline{G}}
%where the sum goes over quotients $K$ of $G \sqcup H_1 \sqcup H_2$, where $H_1$ and $H_2$ are copies of $H$, such that $q(H_1) \cap q(H_2) \subset q(G)$, where $q: G \sqcup H_1 \sqcup H_2 \longrightarrow K$ is the quotient map. 
\end{lemma}

\begin{proof}
Write out the expression for the variance.
$$\operatorname{Var}(\tau_{H \longrightarrow G}) = \E(\tau_{H \longrightarrow \overline{G}}^2) - \E(\tau_{H \longrightarrow \overline{G}})^2$$

The expectation of the square $\E(\tau_{H \longrightarrow \overline{G}}^2)$ is the same  as the expected number of pairs of maps  $H \longrightarrow \overline{G}$, which is the same as the number of maps $H \sqcup H \longrightarrow \overline{G}$.
\end{proof}

We will use the Lemmas \ref{TotalToInjective} and \ref{varianceOfLifts} to count the mean and the variance of the number of the lifts.

\begin{example} \label{CanonicalExample}
% k overloaded (F_k and no. of \gamma)
Suppose $\gamma_1, \ldots , \gamma_k \in F_r$ and $\Gamma_1, \ldots, \Gamma_l < F_k$ and each $\Gamma_j$ has rank at least $2$.
Suppose that $\gamma_i = u_i ^{k_i}$ and that $u_i$ is not a proper power.

Let by abuse of notation $\gamma_i$ be a core graph of $\langle \gamma_i \rangle$. Let $G_j$ be the core graph of $\Gamma_j$.
Let graph $G$ be the disjoint union of $a_i$ copies of $\gamma_i$ and $b_j$ copies of $\Gamma_j$.

Now take a random completion of $G$. We'll count lifts of $\gamma_i$ and $\Gamma_j$.
Let's first calculate $\tau_{\gamma_i \longrightarrow \overline{G}}$. For this we'll need to calculate a contribution from each quotient of $G \cup \gamma_i$. The relative rank $r_G(\gamma_i)$ is at most $ \chi (\gamma_i) = 0$. It can't be smaller, because then there would need to be a component of a critical graph, which is simply connected. That is not possible, because the quotient map is locally injective and $\gamma_i$ contains no leaves.
When counting the critical graphs, two types arise.
\begin{enumerate}
\item The image of $\gamma_i$ is disjoint from all $G$ (we're talking about the additionally copy of $\gamma_i$, not about one of the copies in $G$). There are $\sigma(k_i)$ such quotients, where $\sigma$ counts divisors of an integer.
\item The image of $\gamma_i$ lies in $G$. We can express this quantity as a linear function of $a_j$'s and $b_j$'s.
$$\tau_{\gamma_i \longrightarrow G} = \sum_j a_j \tau_{\gamma_i \longrightarrow \gamma_j} + \sum_j b_j \tau_{\gamma_i \longrightarrow G_j}$$
\end{enumerate}
Use Lemma \ref{Estimate} to get 

$$ \E (\tau_{\gamma_i \longrightarrow \overline{G}}) = \tau_{\gamma_i \longrightarrow G} + \sigma(k_i) + \mathcal{O}(n^{-1})\,.$$

Let's also compute the variance of $\tau_{\gamma_i \longrightarrow \overline{G}}$. Let $H = \gamma_i \sqcup \gamma_i$. By Lemma \ref{varianceOfLifts},

$$\operatorname{Var}(\tau_{\gamma_i \longrightarrow \overline{G}}) =\E (\tau_{H \longrightarrow \overline{G}}) - \E(\tau_{\gamma_i \longrightarrow \overline{G}})^2\,.$$

We have an estimate for the second term, so let's compute the first one.
Again $r_G (H) = 0$. There are four types of quotient contributing to the critical graphs.
\begin{enumerate}
\item Image of $H$ are two circles disjoint from $G$. There are $\sigma(k_i)^2$ such graphs.

\item Both circles of $H$ map to a single circle disjoint from $G$. There are $D(k_i) = \sum_{d|k_i} d$ such graphs as we need to specify the size of the circle and the distance by which are the images of the two circles shifted.

\item One of the circles maps to $G$ and the other remains disjoint. There are $2 \sigma(k_i) \tau_{\gamma_i \longrightarrow G}$ such critical graphs.

\item Both circles map to $G$. There are $\tau_{\gamma_i \longrightarrow G}^2$ such critical graphs.
\end{enumerate}

Add up all these contributions.
\begin{equation*}
\begin{aligned}
\E (\tau_{H \longrightarrow \overline{G}}) &= \sigma(k_i)^2 + D(k_i) + 2 \sigma(k_i) \tau_{\gamma_i \longrightarrow G} + \tau_{\gamma_i \longrightarrow G}^2 + \mathcal{O}(n^{-1})\\
&=(\sigma(k_i) + \tau_{\gamma_i \longrightarrow G})^2 + D(k_i)+\mathcal{O}(n^{-1})
\end{aligned}
\end{equation*}

If we plug it into the expression for variance, most terms cancel out.
\begin{equation*}
\begin{aligned}
\operatorname{Var}(\tau_{\gamma_i \longrightarrow G})&=  (\sigma(k_i) + \tau_{\gamma_i \longrightarrow G})^2 + D(k_i)+\mathcal{O}(n^{-1}) -
(\tau_{\gamma_i \longrightarrow G} + \sigma(k_i) + \mathcal{O}(n^{-1}))^2\\
&= D(k_i)  + \mathcal{O}(n^{-1}).
\end{aligned}
\end{equation*}

Let's now compute the number of lifts of $G_i$. If $b_i \neq 0$, then $\chi_G (G_i) \geq 0$, because we can send $G_i$ to $G$. Also, $\chi_G (G_i) \leq 0$ since no component of a quotient of $G \sqcup G_i$ is simply connected.

Suppose $K$ is a quotient of $G \sqcup G_i$ such that $G \longrightarrow K$ is an injection. Let $L$ be $q(G_i) \setminus q(G)$, where $q$ is the quotient map. There may be open edges in $L$, so it is not necessarily a graph. Then $\chi_G(K) = V(L) - E(L)$. If $K$ is a critical graph, then $V(L) = E(L)$. If $L$ is non-empty, it must contain a component $L'$ with $V(L') \geq E(L')$. The component $L'$ is either a tree, a tree minus a leaf, or a rank $1$ graph. If a component of $L$ is a genuine graph, then it is also a component of $K$. Such a component of $K$ is a locally injective quotient of $G_i$ and therefore has rank at least $2$.  If $L'$ is a tree minus a leaf, then another leaf of $L'$ is a leaf of $K$. This is impossible since all vertices in $G \sqcup G_i \sqcup G_i$ have valence at least $2$ and the quotient map is locally injective.
Therefore $L$ is empty and the critical graphs are precisely the quotients arising from the maps from $G_i$ to $G$. The number of critical graphs is $\tau_{G_i \longrightarrow G}$ and we can use  Lemma \ref{Estimate} to express the expected number of lifts of $G_i$ to a completion of $G$. Therefore,
$$\E(\tau_{G_i \longrightarrow \overline{G}}) = \tau_{G_i \longrightarrow G} + \mathcal{O}(n^{-1}) =   \sum_j b_j \tau_{G_i \longrightarrow G_j} + \mathcal{O}(n^{-1})\,.$$

Similarly, we can compute the variance using Lemma \ref{varianceOfLifts}. We'll need to estimate $\tau_{(G_i \sqcup G_i) \longrightarrow \overline{G}}$. The relative rank of $r_G(G_i \sqcup G_i)$ is at least $0$ because we can send both $G_i$'s to a copy of $G_i$ in $G$. It can't be less, since no component of a quotient of $G \sqcup G_i \sqcup G_i$ is simply connected. 

Suppose $K$ is a critical graph and $L = q(G_i \sqcup G_i) \setminus q(G)$ is non-empty. Then there exists a component $L'$ of $L$, which is either a tree, a tree minus a leaf, or a rank $1$ graph. If $L'$ is a tree or a rank $1$ graph, then it is a component of a  quotient of $G_i \sqcup G_i$. However, the components of quotients of $G_i \sqcup G_i$ have rank at least $2$. If $L'$ is a tree minus a vertex, then another leaf of $L'$  is a leaf of $K$. This is impossible because $K$ is a locally injective quotient of a graph with minimal valence $2$. Therefore $L$ is empty, and $G_i \sqcup G_i$ maps to $G$ in any critical quotient.

There are $(\sum_j b_j \tau_{G_i \longrightarrow G_j})^2$ critical graphs, because we need to specify the images of two copies of $G_i$.

Hence,
$$\E(\tau_{(G_i \sqcup G_i) \longrightarrow \overline{G}})=(\sum_j b_j \tau_{G_i \longrightarrow G_j})^2 + \mathcal{O}(n^{-1})\,.$$

The leading terms cancel out and we are left with a variance that goes to $0$ as $n$ goes to infinity.
$$Var(\tau_{G_i \longrightarrow \overline{G}}) = \E(\tau_{(G_i \sqcup G_i) \longrightarrow \overline{G}})- \E(\tau_{G_i \longrightarrow \overline{G}})^2 = \mathcal{O}(n^{-1})$$
\end{example}

Eventually, the goal is to separate subgroups using distinct numbers of fixed points. In order to do this, we need the following technical lemmas, which promotes groups commensurable to subgroups to actual subgroups. The first lemma says that a core of a finite index subgroup is a cover of a core.

\begin{lemma} \label{FiniteIndex}
Suppose $A,B < F_k$ are finitely generated subgroups and $A$ has a finite index in $B$. Then $\operatorname{Core}(A)$ is a degree $[B:A]$ cover of $\operatorname{Core}(B)$ and in particular $\frac{|V(\operatorname{Core}(A))|}{|V(\operatorname{Core}(B))|} = [B:A]$.
\end{lemma}
\begin{proof}
Let $X_A$ and $X_B$ be the covers of $R_k$ associated to $A$ and $B$. Let $p: X_A \longrightarrow X_B$ be the covering map. Let $d = [B:A]$. Suppose $e \in E(X_A)$ with $p(e) \in \operatorname{Core}(B)$. Then there exists some loop in $\operatorname{Core}(B)$ containing $p(e)$. The $d$-th power of this loop lifts to a loop in $X_A$, which contains $e$, and hence $e \in \operatorname{Core}(A)$.
The restriction $p_{\operatorname{Core}(A)}$ is a local homeomorphism which covers $\operatorname{Core}(B)$ evenly and   $\operatorname{Core}(A)$ is a cover of $\operatorname{Core}(B)$.
\end{proof}

\begin{lemma} \label{SubgroupLifting}
If $H_1, H_2$ are finitely generated subgroups of a free group, $G < H_1 \cap H_2$ has finite index in $H_1$,  and all divisors of $[H_1:G]$ distinct from $1$ are larger than $|V(\operatorname{Core}(H_2))|$, then $H_1$ is a subgroup of  $H_2$.
\end{lemma}

\begin{proof}
Consider $\operatorname{Core}(H_1 \cap H_2)$. We can get it as a component of the pullback of the maps $\operatorname{Core}(H_i) \longrightarrow X$, where $X$ is the rose $R_k$. The pullback contains $|V(\operatorname{Core}(H_1)| |V(\operatorname{Core}(H_2)|$ vertices, therefore
$$|V(\operatorname{Core}(H_1 \cap H_2)| \leq |V(\operatorname{Core}(H_1)| |V(\operatorname{Core}(H_2)|\,.$$
The group $G$ is a finite index subgroup of $H_1 \cap H_2$, which is a finite index subgroup of $H_1$. By lemma \ref{FiniteIndex} applied to $H_1 \cap H_2 < H_1$ and to $G <H_1 \cap H_2$, $|V(\operatorname{Core}(H_1))|$ divides $|V(\operatorname{Core}(H_1 \cap H_2)|$, which divides $|V(\operatorname{Core}(G))|$. Then $|V(\operatorname{Core}(H_1 \cap H_2)|= d|V(\operatorname{Core}(H_1)|$, where $d$ divides $\frac{|V(\operatorname{Core}(G))|}{|V(\operatorname{Core}(H_1))|}= [H_1 :G]$. But every nontrivial divisor of $[H_1 : G]$ is larger than $|V(\operatorname{Core}(H_2))|$, so $|V(\operatorname{Core}(H_1 \cap H_2)| = |V(\operatorname{Core}(H_1)|$.
 Since $\operatorname{Core}(H_1 \cap H_2)$ is a covering of $\operatorname{Core}(H_1)$, the two graphs are in fact equal and $H_1 \cap H_2 = H_1$.
\end{proof}

Finally, we can put everything together in the proof of the following separability property, which can be thought of as an `alternating' refinement of subgroup into-conjugacy separability. We will do this by using that whenever $H_1$ is conjugate into $H_2$, it fixes at least as many elements as $H_2$. We will also use that the same is true for concrete characteristic subgroups.
For example suppose $H_1$ is not conjugate into $H_2$ and $H_2$ is not conjugate into $H_1$. If $H_1$ fixes more points than $H_2$, then $f(H_2)$ is not conjugate into $f(H_1)$. If additionally the intersection of all degree $2$ subgroups of $H_2$ fixes more points than the intersection of all degree $2$ subgroups of $H_1$, then $f(H_1)$ is not conjugate into $f(H_2)$.

\begin{theorem} \label{MainTheoremOfTheChapter3}
%Why is it F_r and not F_k?
Suppose $H_1, H_2, \ldots , H_n < F_r$ are finitely generated subgroups of infinite index. Then there exists a surjective homomorphism $f: F_r \twoheadrightarrow A_m$ such that whenever $H_i$ is not conjugate into $H_j$, then $f(H_i)$ is not conjugate into $f(H_j)$.
\end{theorem}

\begin{proof}
Denote the relation of `is conjugate into' by `$\prec$'.
Conjugacy classes of finitely generated subgroups of $F_r$ form a poset with respect to $\prec$ so after reordering and removing duplicates, we may assume that $H_i \prec H_j$ implies $i \leq j$.

Let $p_1, p_2, \ldots, p_n$ be primes larger than $\max_i (V(\operatorname{Core}(H_i))$ with $p_j > p_k^{(k!)^{rk{H_k}}} V(\operatorname{Core}(H_k))$ whenever $j<k$. Let $G_{i,j}$ be the intersection of all index $p_j$ subgroups of $H_i$.
Let graph $G$ be a union of $a_i$ copies of $\operatorname{Core}(G_{i,i})$, where $a_i$'s are to be specified later. Let $f: F_r \longrightarrow A_m$ be a random map arising from a random completion of $G$. The group $f(G_{i,j})$ is the intersection of all index $p_j$ subgroups of $f(H_i)$. Indeed, every index $p_j$ subgroup of $f(H_i)$ is an image of an index $p_j$ subgroup of $H_i$.

If $f(H_i) \prec f(H_j)$, then $fix(f(H_i)) \geq fix(f(H_j))$, but also $f(G_{i,k}) \prec f(G_{j,k})$ and hence $fix(G_{i,k}) \geq fix(G_{j,k})$.

By Example \ref{CanonicalExample} for every $\varepsilon$ there exists $K=K(\varepsilon)$ independent of $a_1, \ldots , a_n$ such that for all sufficiently large $m$ 
\begin{equation} \label{Estimate2}
\Pb( \forall i,j, |fix (G_{i,j}) - \Sigma_k a_k \tau_{\operatorname{Core}(G_{i,j}) \longrightarrow \operatorname{Core}(G_{k,k})}| < K) > 1 - \varepsilon
\end{equation} 
In words, the number of fixed points of $G_{i,j}$ belongs with high probability to a specific interval of length $2K$. By controlling the center of the interval, we will ensure that these groups often fix distinct numbers of elements.

If $\tau_{\operatorname{Core}(G_{i,j}) \longrightarrow \operatorname{Core}(G_{k,k})} > 0$, then $G_{i,j} < G_{k,k}^g$ for some $g$. Both $H_i$ and $G_{k,k}^g$ are subgroups of a free group, and the  index of $G_{i,j} < H_i \cap G_{k,k}^g$ in $H_i$ is a power of $p_j$. The core of $G_{k,k}$ contains at most $p_k^{(k!)^{rk{H_k}}} V(\operatorname{Core}(H_k))$ vertices. If $j<k$, then $p_j > p_k^{(k!)^{rk{H_k}}} V(\operatorname{Core}(H_k))$ and by Lemma \ref{SubgroupLifting} $H_i < G_{k,k}^g$. This is a contradiction since the girth of $\operatorname{Core}(H_i)$ is at most $V(\operatorname{Core}(H_i))$ and the girth of $\operatorname{Core}(G_{k,k})$ is at least $p_k > V(\operatorname{Core}(H_i))$.

We also have have $p_j > V(\operatorname{Core}(H_k))$, so Lemma \ref{SubgroupLifting} applied to $H_i, G_{i,j}$ and $H_k^g$ gives that $H_i \prec H_k$.

%
%$ fix (G_{i,j}) \simeq \sum_{k:k \leq j \land \and H_i \prec H_k} a_k \tau_{\operatorname{Core}(G_{i,j}) \longrightarrow \operatorname{Core}(G_{k,k})}  $

Let $K$ be such that the probability in Equation \ref{Estimate2} is at least $p = 1- 2^{-r-1}$. Let $a_1, \ldots , a_n$ satisfy $a_j > na_{j-1} C + K$, where $C = \max_{i,j,k} \tau_{\operatorname{Core}(G_{i,j}) \longrightarrow \operatorname{Core}(G_{k,k})} $.

All of the following is simultaneously true with probability at least $1- 2^{-r-1}$. For all $j$, $fix(G_{j,j}) \geq a_j$. For all $i,j$, if $H_i$ is not conjugate into $H_j$, then $\operatorname{fix}(G_{i,j}) \leq \max(0, (j-i)a_{j-1} C + K)< a_j$. Hence  $f(H_i)$ is not conjugate into $f(H_j)$.

%If $a_k = C^{n-k}$, then $fix(G_{k,k})$ is about $ C^{n-k}$, which is more than the estimate for  $fix(G_{i,j})$ whenever $i > k$. If we instead take $a_k = \alpha C^k$ and $\alpha$ very large, we can ensure that $fix(G_{k,k}) < fix(G_{i,k})$ whenever $i>k$ with arbitrarily large probability $p$. In such a situation $f(G_{i,k})$ is not conjugate into $f(G_{k,k})$ and hence $f(H_i)$ is not conjugate into $f(H_k)$.

The probability that the image is $A_m$ tends to $2^{-r}$ as $m$ goes to infinity (Theorem \ref{MainTheorem}). In particular, there exists a map $f$ with the described separating properties.
\end{proof}

\bibliographystyle{alpha}
\bibliography{../mybib}

\end{document}